\documentclass[a4paper,12pt,twoside]{amsart}
\usepackage[english]{babel}
\usepackage{amssymb, amsmath, eucal}
\usepackage[all]{xy}
\usepackage{mathrsfs}
\usepackage[dvips]{graphics}

\newtheorem{The}{Theorem}[section]
\newtheorem{Lem}[The]{Lemma}
\newtheorem{Prop}[The]{Proposition}
\newtheorem{Cor}[The]{Corollary}

\newtheorem{Prob}[The]{Problem}
\newtheorem{Def}[The]{Definition}

\newcommand{\K}{\mathcal{K}}
\newcommand{\C}{\mathbb{C}}
\newcommand{\R}{\mathbb{R}}

\newcommand{\Z}{\mathbb{Z}}

\newcommand{\E}{\mathcal{E}}
\newcommand{\F}{\mathcal{F}}

\begin{document}
 \title[ Monge-Amp\`ere equation]{Complex Monge-Amp\`ere equation in strictly pseudoconvex domains}

\setcounter{tocdepth}{1}

  \author{Hoang-Son Do} 
\address{Institute of Mathematics \\ Vietnam Academy of Science and Technology \\18
Hoang Quoc Viet \\Hanoi \\Vietnam}
\email{hoangson.do.vn@gmail.com , dhson@math.ac.vn}
\author{Thai Duong Do}
\address{Institute of Mathematics \\ Vietnam Academy of Science and Technology \\18
	Hoang Quoc Viet \\Hanoi \\Vietnam}
\email{dtduong@math.ac.vn}
\author{Hoang Hiep Pham}
\address{Institute of Mathematics \\ Vietnam Academy of Science and Technology \\18
	Hoang Quoc Viet \\Hanoi \\Vietnam}
\email{phhiep@math.ac.vn}
 \date{\today\\ This research is funded by Vietnam National Foundation for Science and Technology Development (NAFOSTED) under grant number 101.02-2017.306.}


\begin{abstract}
  We study the complex Monge-Amp\`ere equation $(dd^c u)^n=\mu$ in a strictly pseudoconvex domain $\Omega$ with the boundary condition $u=\varphi$, where
  $\varphi\in C(\partial\Omega)$. We provide a non-trivial sufficient condition for continuity of the solution $u$ outside 
  ``small sets''. 
\end{abstract}

\maketitle
\begin{center}
	{\it In honor of L\^e V\u{a}n Thi\^em's centenary} 
\end{center}
\tableofcontents
\newpage

\section*{Introduction}
 Let $\Omega\subset\C^n$ be a bounded domain. In this paper, we  always assume that $\mu$ is a Borel probability measure in
 $\Omega$  and suppose that $\varphi$ is a continuous function in $\partial\Omega$.
 The following result has been proven by Kolodziej \cite{Kol98}.
\begin{The}\label{Kol.the}
	Assume that $\Omega$ is strictly pseudoconvex. Consider an increasing function $h:\R\rightarrow (1, \infty)$ satisfying
	\begin{equation}\label{conditionh.eq}
	\int\limits_1^{\infty}(y h^{1/n}(y))^{-1}dy<\infty.
	\end{equation}
	If $\mu$ satisfies the inequality
	\begin{equation}\label{conditionmu.eq}
		\mu (K)\leq A Cap (K, \Omega) h^{-1}((Cap (K, \Omega))^{-1/n}),
	\end{equation}
for any $K\subset\Omega$ compact and regular then there exists a unique $u\in PSH(\Omega)\cap C(\overline{\Omega})$ such that
\begin{equation}\label{MAE}
\begin{cases}
(dd^c u)^n=\mu\ \mbox{in}\ \Omega,\\
u=\varphi\ \mbox{in}\ \partial\Omega.
\end{cases}
\end{equation}
Moreover, $\|u\|_{L^{\infty}}$ is bounded by a constant $B=B(h, A, \varphi, \Omega)$ which does not depend on $\mu$.
\end{The}
Here, $\Omega$ is called  strictly pseudoconvex domain iff
\begin{center}
	$\Omega=\{z\in\C^n:\rho (z)<0 \}$,
\end{center} 
where $\rho\in C^2(\overline{\Omega})$ satisfies $\nabla\rho (z)\neq 0$ for every $z\in\partial\Omega$ and 
$dd^c \rho\geq a\omega=:a dd^c |z|^2$ in $\Omega$ for some $a>0$.

Some plurisubharmonic functions are not continuous in the whole $\Omega$ even though they are continuous outside an analytic set. For example,
$u=-(-\log |z|)^{1/2}$ is not continuous in the whole unit ball $B$, but it is continuous in 
$B\setminus\{0\}$. We are interested in the following problem.
\begin{Prob}
	 Find conditions for $\mu$ such that $u$ is continuous outside an analytic set $E$ but $u$ may not be continuous in $\Omega$.
\end{Prob}
The purpose of this article is to provide a sufficient condition for the continuity of $u$ outside ``small sets'' and consequently,
provide a sufficient condition for the continuity of $u$ outside an analytic set. Our main result is the following.
\begin{The}\label{main}
	Suppose that $\Omega$ is strictly pseudoconvex and $v\in\E (\Omega)$. 
Assume that there exists a sequence $\{M_j\}_{j=1}^{\infty}$ of postive real numbers with $\lim\limits_{j\to\infty}M_j=\infty$
such that
\begin{itemize}
	\item [(i)] For any $j\in\Z^+$, $\chi_{U_j}\mu\leq \dfrac{1}{2^j}\chi_{U_j}(dd^c v)^n$, where 
	$U_j=\{z\in\Omega| v(z)<-M_j\}$.
	\item[(ii)] For any $j\in\Z^+$, there exist $h=h_j, A=A_j$ satisfying \eqref{conditionh.eq} and 
	\eqref{conditionmu.eq} for every compact set $K\subset V_j:=\Omega\setminus U_j$.	
\end{itemize}
Then, there exists a unique function $u$ satisfying
\begin{equation}\label{GMAE}
\begin{cases}
u\in \F(\varphi, \Omega),\\
(dd^cu)^n=\mu,\\
u\in C(V_j), \forall j\in\Z^+.
\end{cases}
\end{equation}
Moreover, for each $j\in\Z^+$, for any $z\in\partial\Omega\cap\overline{V}_j$,
\begin{center}
	$\lim\limits_{V_j\ni\xi\to z}u(\xi)=\varphi (z).$
\end{center}
\end{The}
\begin{Cor}\label{main.cor}
Assume that the assumption of Theorem \ref{main} is satisfied. If 
	there exist $\alpha\in (0, 1)$, $\lambda_1, ..., \lambda_m>0$ and analytic functions $f_1, ..., f_m\in\mathcal{A}(\C^n)$
	such that $v=-(-\log (|f_1|^{\lambda_1}+...+|f_m|^{\lambda_m}))^{\alpha}$ in $\Omega$ then 
	$u\in C(\Omega\setminus F)$, where $F=\{f_1=f_2=...=f_m=0\}$. Moreover, for any $z\in\partial\Omega\setminus F$,
	\begin{center}
		$\lim\limits_{\Omega\setminus F\ni\xi\to z}u(\xi)=\varphi (z).$
	\end{center}
\end{Cor}

\section{Preliminaries}
For the convenience, throughout this section, unless otherwise specified, we assume that $\Omega$ is a hyperconvex domain.
\subsection{Some classes of plurisubharmonic functions}
We recall the definition of some classes of plurisubharmonic functions. The following classes were first introduced by Cegrell.
\cite{Ceg98, Ceg04}
\begin{Def} 
$\E_0(\Omega)=\{u\in PSH(\Omega)\cap L^\infty (\Omega): \lim_{z\to\partial\Omega}u(z)=0, \int_\Omega (dd^c u)^n<\infty\},$\\
$\F(\Omega)=\{u\in PSH(\Omega): \exists \, \{u_j\}\subset \E_0(\Omega), \; u_j\searrow u, \, \sup_{j}\int_{\Omega} 
(dd^cu_j)^n<\infty\},$\\
$\E(\Omega)=\{u\in PSH(\Omega): \forall\, \omega\Subset \Omega \,\, \exists\, u_{\omega}\in \F(\Omega) \text{ such that } u_{\omega}=u \text { on } \omega \},$\\
$\mathcal{N}(\Omega)=\{u\in\E(\Omega):$ the smallest maximal plurisubharmonic majorant $=0 \}$.
\end{Def}
It is clearly that $\E_0\subset\F\subset\mathcal{N}\subset\E$. In the case where $\Omega$ is hyperconvex, Cegrell has shown
that $\E(\Omega)=\mathcal{D}(\Omega)\cap PSH^-(\Omega)$, 
where $\mathcal{D}(\Omega)$ is the biggest subclass of PSH($\Omega$) 
where the Monge-Amp\`ere operator is 
well-defined. In the general case, there is a charateristic of $\mathcal{D}(\Omega)$ introduced by Blocki \cite{Blo06}.

The classes $\E_0, \F, \mathcal{N}$ can be generalized as following 
(see \cite{Ceg98}, \cite{Aha07}).
\begin{Def}\label{def2} 
	i) Let $H$ be a maximal plurisubharmonic function in $\Omega$. For $\K\in\{\E_0, \F, \mathcal{N}\}$, we denote
	\begin{center}
		$\K (H)=K(H, \Omega)=\{u\in PSH(\Omega): \exists\phi\in\K, H\geq u\geq \phi+H\}.$
	\end{center}	
	ii)  Let $f\in C(\partial\Omega)$. For $\K\in\{\E_0, \F, \mathcal{N}\}$, we denote
	\begin{center}
		$\K (f)=K(f, \Omega)=\{u\in PSH(\Omega): \exists\phi\in\K, U(0, f)\geq u\geq \phi+U(0, f)\},$
	\end{center}
where $U(0, f)$ is the unique solution of
\begin{equation}\label{max.eq}
\begin{cases}
	U(0, f)\in PSH (\Omega)\cap L^{\infty}(\overline{\Omega}),\\
	(dd^c U(0, f))^n=0\ \mbox{in}\ \Omega,\\
	U(0, f)=f\ \mbox{in}\ \partial\Omega.
\end{cases}
\end{equation}
\end{Def}
Note that, if $\Omega$ is strictly pseudoconvex then $U(0, f)$ is always continuous \cite{BT76}.
\subsection{The solution of Monge-Amp\`ere equation}
The complex Monge-Amp\`ere equation is an important object in the pluripotential theory
with many interesting related results. We recall some useful results for proving our main theorem. 
\begin{The}\cite{Ceg04} If $\mu$ vanishes on all pluripolar sets then there exists a unique solution $u\in\F (\Omega)$ of the 
	equation $(dd^c u)^n=\mu$.
\end{The}
\begin{The}\label{Ahag.the} \cite{Aha07}
	If $\mu$ vanishes on all pluripolar sets and $U(0, f)\in C(\overline{\Omega})$
	then there exists a unique solution $u\in\F (\varphi, \Omega)$ of the 
	equation $(dd^c u)^n=\mu$.
\end{The}
\begin{The}\cite{ACCP09}
	If there exists a function $v\in\E (\Omega)$ such that $(dd^c v)^n\geq\mu$ then for every maximal plurisubharmonic function
	$H\in\E(\Omega)$, there exists a function $u\in\E(\Omega)$ such that $v+H\leq u\leq H$ and $(dd^cu)^n=\mu$.
\end{The}
\subsection{Comparison principles} The idea of comparison principles is to use
the comparison between the Monge-Amp\`ere operators of two plurisubharmonic functions $u, v$
to compare $u$ and $v$.

The Bedford-Taylor comparison principle is the following.
\begin{The}\label{the.compa}\cite{BT82}
	 Let $u, v\in PSH(\Omega)\cap L^{\infty}(\Omega)$ such that
	\begin{center}
		$\liminf\limits_{\Omega\ni z\to\partial \Omega}(u(z)-v(z))\geq 0$.
	\end{center}
	Then
	\begin{center}
		$\int\limits_{\{u<v\}}(dd^c v)^n\leq\int\limits_{\{u<v\}}(dd^c u)^n.$
	\end{center}
\end{The}
Theorem \ref{the.compa} has been generalized in several directions. One of improved versions is
the following
\begin{The}\label{the.npcompa}\cite{NP09, ACCP09}
 Let $u, v\in \mathcal{E}(\Omega)$. Assume that one of the following conditions holds
\begin{itemize}
	\item [(i)] $\liminf\limits_{\Omega\ni z\to\partial \Omega}(u(z)-v(z))\geq 0$.\\
	\item [(ii)] $u\in\mathcal{N}(H, \Omega)$ for some maximal plurisubharmonic function $H\leq 0$, and $v\leq H$.
\end{itemize}
	Then, 
	\begin{center}
		$\dfrac{1}{n!}\int\limits_{\{u<v\}}(v-u)^ndd^cw_1\wedge...\wedge dd^cw_n
		+\int\limits_{\{u<v\}}-w_1(dd^cv)^n
		\leq\int\limits_{\{u<v\}\cup\{u=v=-\infty\}}-w_1(dd^cu)^n,$
	\end{center}
	for any  $w_1,...,w_n\in PSH(\Omega, [-1,0])$.
\end{The}
The following corollary of Theorem \ref{the.npcompa} will be used to prove the main theorem.
\begin{The}\label{the.compaaccp}\cite{ACCP09}
	Let $u, v\in\E (\Omega)$ such that $(dd^cu)^n$ vanishes on all pluripolar sets and $(dd^c u)^n\leq (dd^cv)^n$. 
	 Assume that one of the following conditions holds
	\begin{itemize}
		\item [(i)] $\liminf\limits_{\Omega\ni z\to\partial \Omega}(u(z)-v(z))\geq 0$.\\
		\item [(ii)] $u\in\mathcal{N}(H, \Omega)$ for some maximal plurisubharmonic function $H\leq 0$, and $v\leq H$.
	\end{itemize}
Then $u\geq v$ in $\Omega$.
\end{The}
\subsection{Relative capacity}
Let $K$ be a compact subset of $\Omega$. The relative capacity of $K$ in $\Omega$ is defined by
\begin{center}
	$Cap (K, \Omega)=\sup\{\int\limits_K(dd^c v)^n| v\in PSH(\Omega, [0, 1])\}.$
\end{center}
If $E\subset\Omega$ then the relative capacity of $K$ in $\Omega$ is defined by
\begin{center}
	$Cap (E, \Omega)=\sup\{Cap (K, \Omega)| K\ \mbox{is a compact subset of}\ E \}$.
\end{center}
\begin{Prop}
	If $E\subset\Omega$ is a Borel set then
	\begin{center}
	$Cap (E, \Omega)=\sup\{\int\limits_E(dd^c v)^n| v\in PSH(\Omega, [0, 1])\}.$
	\end{center}
\end{Prop}
\begin{Prop}
	If $E$ is a pluripolar set then $Cap (E, \Omega)=0$.
\end{Prop}
We refer the reader to \cite{BT82}, \cite{Kli91}, \cite{Kol05} for more properties of the
relative capacity.
\section{Proof of the main theorem}
First, we show that $\mu$ vanishes on all pluripolar sets. Let $F\subset\Omega$ be an arbitrary pluripolar set. 
For every compact set $K\subset\Omega$ and for every $j\in\Z^+$, we have
\begin{center}
	$\mu (F\cap K\cap V_j)\leq A_j Cap (F\cap K\cap V_j, \Omega) h_j^{-1}((Cap (F\cap K\cap V_j, \Omega))^{-1/n})=0,$
\end{center}
and
\begin{center}
	$\mu (F\cap K\cap U_j)\leq \dfrac{1}{2^j}\int\limits_{F\cap K\cap U_j}(dd^c v)^n\leq \dfrac{1}{2^j}\int\limits_K(dd^c v)^n.$
\end{center}
Hence,
\begin{center}
	$\mu (F\cap K)=\mu (F\cap K\cap U_j)+\mu (F\cap K\cap V_j)\leq \dfrac{1}{2^j}\int\limits_K(dd^c v)^n.$
\end{center}
Letting $j\rightarrow\infty$, we get
\begin{center}
		$\mu (F\cap K)=0.$
\end{center}
Since $K$ is arbitrary, we have $\mu (F)=0$. Then, $\mu$ vanishes on all pluripolar sets.\\

Now, by using Theorem \ref{Ahag.the}, there exists a unique function $u$ satisfying
\begin{center}
	$\begin{cases}
		u\in \F(\varphi, \Omega),\\
		(dd^cu)^n=\mu.
	\end{cases}$
\end{center}
It remains to show that $u\in C(V_j\cup \partial\Omega)$  if we define $u=\varphi$ in $\partial\Omega$. 

 By Theorem \ref{Kol.the}, for any $j\in\Z^+$, there exists a unique solution $u_j$ of the equation
 \begin{equation}
 \begin{cases}
 (dd^c u_j)^n=\chi_{V_j}\mu\ \mbox{in}\ \Omega,\\
 u=\varphi\ \mbox{in}\ \partial\Omega.
 \end{cases}
 \end{equation}
 It is easy to check that
 \begin{center}
 	$(dd^c u_j)^n\leq (dd^c u)^n\leq (dd^c(u_j+\dfrac{v}{2^{j/n}}))^n,$
 \end{center}
for every $j\in\Z^+$.

By Theorem \ref{the.compaaccp}, we have
\begin{equation}\label{squeeze.eq}
	u_j+\dfrac{v}{2^{j/n}}\leq u\leq u_j,
\end{equation}
for every $j\in\Z^+$.

Let $j_0$ be an arbitrary positive integer. For any $\epsilon>0$, there exists $j\gg 1$ such that
\begin{equation}\label{estv.eq}
	\dfrac{M_{j_0}}{2^{j/n}}<\dfrac{\epsilon}{2}.
\end{equation}
By \eqref{squeeze.eq}, \eqref{estv.eq} and by $u|_{\partial\Omega}=\varphi$, we have
\begin{equation}\label{squeeze2.eq}
u_j-\dfrac{\epsilon}{2}\leq u\leq u_j,
\end{equation}
in $V_{j_0}\cup\partial\Omega$.

By the continuity of $u_j$, there exists $\delta>0$ such that, 
\begin{equation}\label{estuj.eq}
	|u_{j}(z)-u_{j}(w)|<\dfrac{\epsilon}{2},
\end{equation}
for all $z, w\in\overline{\Omega}, |z-w|<\delta$.

Combining \eqref{squeeze2.eq} and \eqref{estuj.eq}, we get
\begin{center}
	$u(z)-u(w)\leq u_j(z)-(u_j(w)-\dfrac{\epsilon}{2})<\epsilon, $
\end{center}
for all $z, w\in V_{j_0}\cap\partial\Omega, |z-w|<\delta$.\\
Hence, $u\in C(V_{j_0}\cup\partial\Omega)$.
\section{A remark on the class $\E (\Omega)$}
In this section, we discuss the condition 
\textit{``$v=-(-\log (|f_1|^{\lambda_1}+...+|f_m|^{\lambda_m}))^{\alpha}\in\E (\Omega) (=\mathcal{D}(\Omega)\cap PSH^-(\Omega))$''}
in Corollary \ref{main.cor}. If $0<\alpha<\frac{1}{n}$ then $v\in\E (\Omega)$ (see \cite{Bed93}, \cite{Blo09}).
 In the case where $F$ is non-singular, the set $\{\alpha\in (0, 1)| v\in\E (\Omega)\}$ can be clearly described as the following
\begin{Prop}\label{proposition}
	Let $\Omega\subset\C^n$ be a bounded domain and $\lambda_1, ...,
	\lambda_m>0$. Let $f_1,..., f_m\in\mathcal{A} (\Omega)$ such that
	$|f_1|^{\lambda_1}+...+|f_m|^{\lambda_m}<1$ in $\Omega$.
	 Assume that $F=\{f_1=...=f_m=0\}$ is non-singular
	and $n>dim_{\C}F=n-k>0$. 
	Then $v=-(-\log (|f_1|^{\lambda_1}+...+|f_m|^{\lambda_m}))^{\alpha}\in\mathcal{D} (\Omega)$ iff $\alpha\in (0, \frac{k}{n})$.
\end{Prop}
Note that if $dim_{\C}F=0$ (i.e. $F$ is a finite set) then
$\log (|f_1|^{\lambda_1}+...+|f_m|^{\lambda_m})\in\mathcal{D} (\Omega)$. As a consequence, 
$v\in\mathcal{D} (\Omega)$ for any $\alpha\in (0, 1)$.

In order to prove Proposition \ref{proposition}, we need the following lemma.
\begin{Lem}\label{lemeomega}
	Let $0<k<n$. In the ball $B=\{z\in\C^n: |z|<1/2\}$, consider the plurisubharmonic functions
	\begin{center}
		$u_{\alpha \epsilon}=-(-\log(|z_1|^2+...+|z_k|^2+\epsilon))^{\alpha},$
	\end{center}
	where $\epsilon\in (0, 1/2), \alpha\in (0, 1)$. Then, 
	\begin{center}
		$\limsup\limits_{\epsilon\to 0}\int\limits_{B}|u_{\alpha \epsilon}|^{n-p-2}
		du_{\alpha\epsilon}\wedge d^cu_{\alpha\epsilon}\wedge (dd^cu_{\alpha\epsilon})^p\wedge\omega^{n-p-1}<\infty$,
	\end{center}
	for any $p=0, 1,...,n-2$ iff $\alpha<\frac{k}{n}$. Here $\omega=dd^c |z|^2$.
\end{Lem}
\begin{proof}
	For $\epsilon\in (0, 1/2)$, we denote
	\begin{center}
		$u_{\epsilon}(z)=\log (|z_1|^2+...+|z_k|^2+\epsilon).$
	\end{center}
Then, for any $\alpha\in (0, 1)$,
\begin{flushleft}
	$u_{\alpha\epsilon}=-(-u_{\epsilon})^{\alpha},$\\[14pt]
	$d u_{\alpha\epsilon}=\alpha (-u_{\epsilon})^{\alpha-1}d u_{\epsilon}$,\\[14pt]
	$d u_{\alpha\epsilon}\wedge d^c u_{\alpha\epsilon}
	=\alpha^2 (-u_{\epsilon})^{2(\alpha-1)}d u_{\epsilon}\wedge d^c u_{\epsilon}$,\\[14pt]
	$dd^c u_{\alpha\epsilon}= 
	\alpha (-u_{\epsilon})^{\alpha-1}dd^c u_{\epsilon}
	+\alpha (1-\alpha)(-u_{\epsilon})^{\alpha-2}d u_{\epsilon}\wedge d^c u_{\epsilon}$.\\[14pt]
\end{flushleft}
Since $u_{\alpha\epsilon}$ depends only on $k$ variables,
 $d u_{\alpha\epsilon}\wedge d^c u_{\alpha\epsilon}
 \wedge (dd^c u_{\alpha\epsilon})^p=0$ for any $p\geq k$. For any $p=0, 1, ..., k-1$, 
 we have
 \begin{equation}\label{aeblocki.eq}
 	d u_{\alpha\epsilon}\wedge d^c u_{\alpha\epsilon}
 	\wedge (dd^c u_{\alpha\epsilon})^p
 	=\alpha^{p+1}|u_{\epsilon}|^{\alpha n-p-2}d u_{\epsilon}\wedge d^c u_{\epsilon}
 	\wedge (dd^cu_{\epsilon})^p.
 \end{equation}
 By calculating, we have 
 \begin{center}
 	$\left(\dfrac{\partial u_{\epsilon}}{\partial z_j}
 	\dfrac{\partial u_{\epsilon}}{\partial \bar{z}_l}\right)_{j,l=\overline{1,k}}
 		=e^{-2u_{\epsilon}}\left( \bar{z}_jz_l \right)_{j,l=\overline{1,k}}=:A$
 		and
 	$\left(\dfrac{\partial^2 u_{\epsilon}}{\partial z_j\partial\bar{z}_l}\right)_{j,l=\overline{1,k}}
 	=e^{-u_{\epsilon}}Id_k-A$.
 \end{center}
Since $rank A\in\{0, 1\}$, there exists a $k\times k$ unita matrix $U$ such that
\begin{center}
	$A=U^* diag(e^{-2u_{\epsilon}}(|z_1|^2+...+|z_k|^2), 0, ..., 0)U$,
\end{center}
and then
\begin{center}
	$e^{-u_{\epsilon}}Id_k-A=
	e^{-u_{\epsilon}}U^* diag(\epsilon e^{-u_{\epsilon}}, 1, ..., 1)U$.
\end{center}
Then,  for any $p=0, 1, ..., k-1$, we have
\begin{equation}\label{calloge.eq}
	d u_{\epsilon}\wedge d^c u_{\epsilon}
\wedge (dd^c u_{\epsilon})^p\wedge\omega^{n-p-1}
=C(n, p)e^{-(p+2)u_{\epsilon}}(|z_1|^2+|z_2|^2+...+|z_k|^2)dV_{2n},
\end{equation}
where $C(n, p)>0$ depends only on $n$ and $p$.

By combining \eqref{aeblocki.eq}, \eqref{calloge.eq} and by Fubini's theorem, we have
\begin{flushleft}
	$\int\limits_Bd u_{\alpha\epsilon}\wedge d^c u_{\alpha\epsilon}
	\wedge (dd^c u_{\alpha\epsilon})^p\wedge\omega^{n-p-1}$\\
	$\sim \int\limits_{\{|z_1|^2+...+|z_k|^2<1/4\}}
	\dfrac{|u_\epsilon|^{n\alpha-p-2}
		(|z_1|^2+|z_2|^2+...+|z_k|^2)dV_{2k}}
	{e^{(p+2)u_{\epsilon}}}$\\
	$\sim\int\limits_0^{1/2}
	\dfrac{(-\log (t^2+\epsilon))^{n\alpha-p-2}t^2t^{2k-1}dt}{(t^2+\epsilon)^{p+2}}
	=\int\limits_0^{1/2}
	\dfrac{(-\log (t^2+\epsilon))^{n\alpha-p-2}t^{2k+1}dt}{(t^2+\epsilon)^{p+2}},$
\end{flushleft}
for $p=0, ..., k-1$. Here, $A\sim B$ means that $c_1A\leq B\leq c_2A$, where $c_1, c_2>0$
are independent on $\epsilon, \alpha$.

For any $p\leq k-2$, we have
\begin{center}
	$\int\limits_0^{1/2}\dfrac{(-\log(t^2+\epsilon))^{n\alpha-p-2}
		t^{2k+1}dt}{(t^2+\epsilon)^{(p+2)}}
	\leq \int\limits_0^{1/2}(-\log(t^2))^{n-2}
		tdt<\infty$,
\end{center}
for every $\alpha\in (0, 1), \epsilon\in (0, 1/2)$.

If $p=k-1$ then
\begin{center}
	$	\int\limits_0^{1/2}\dfrac{(-\log(t^2+\epsilon))^{n\alpha-p-2}
	t^{2k+1}dt}{(t^2+\epsilon)^{(p+2)}}
=\int\limits_0^{1/2}\dfrac{(-\log(t^2+\epsilon))^{n\alpha-k-1}
t^{2k+1}dt}{(t^2+\epsilon)^{(k+1)}}.$
\end{center}
If $0<\alpha<\dfrac{k}{n}$ then, for any $0<\epsilon<1/3$,
\begin{flushleft}
	$\begin{array}{ll}
	\int\limits_0^{1/2}\dfrac{(-\log(t^2+\epsilon))^{n\alpha-k-1}
		t^{2k+1}dt}{(t^2+\epsilon)^{(k+1)}}
	&\leq \int\limits_0^{1/2}\dfrac{(-\log(t^2+\epsilon))^{n\alpha-k-1}dt}
	{(t^2+\epsilon)^{1/2}}\\
	&\leq 2^{n\alpha-k-1/2}\int\limits_0^{1/2}\dfrac{(-\log(t+\epsilon))^{n\alpha-k-1}dt}
	{(t+\epsilon)}\\
	&\leq 2^{n\alpha-k-1/2}\int\limits_0^{5/6}\dfrac{(-\log t)^{n\alpha-k-1}dt}{t}\\
	&<\infty.
	\end{array}$
\end{flushleft}
If $\alpha\geq\dfrac{k}{n}$ then, by Fatou's lemma,
\begin{flushleft}
	$\liminf\limits_{\epsilon\to 0}
	\int\limits_0^{1/2}\dfrac{(-\log(t^2+\epsilon))^{n\alpha-k-1}
		t^{2k+1}dt}{(t^2+\epsilon)^{(k+1)}}
	\geq \int\limits_0^{1/2}\dfrac{(-\log(t^2))^{n\alpha-k-1}dt}
	{t}=\infty.$
\end{flushleft}
This completes the proof.
\end{proof}
By \cite{Blo06} and Lemma \ref{lemeomega}, we have the following.
\begin{Cor}\label{cor}
		Let $0<k<n$. In the ball $B=\{z\in\C^n: |z|<1/2\}$, consider the plurisubharmonic functions
	\begin{center}
		$u_{\alpha}=-(-\log(|z_1|^2+...+|z_k|^2))^{\alpha},$
	\end{center}
	where $\alpha\in (0, 1)$.
	 Then, $u_{\alpha}\in\mathcal{D}(B)$ iff $0<\alpha<\frac{k}{n}$.\\[14pt]
\end{Cor}
Now, we prove the proposition \ref{proposition}.
\begin{proof}[Proof of Proposition \ref{proposition}]
Let $a\in\Omega$. If $a\notin F$ then there exists an open neighbourhood $U$ of $a$ such that
$v$ is bounded in $U$ (and then $v\in\mathcal{D} (U)$).

If $a\in F$ then there exist an open neighbourhood $U$ of $a$ and a 
biholomorphic function $\phi: U\rightarrow B=\{z\in\C^n: |z|<1/2\}$ such that
$\phi (U\cap F)=\{z\in B | z_1=...=z_k=0\}$. 
By  Hilbert's Nullstellensatz theorem (see, for example, \cite[p.19]{Huy05}), 
there exist $M, N>0$ such that 
\begin{center}
$(f_1\circ\phi^{-1})^M, ...,(f_m\circ \phi^{-1})^M\in\langle z_1,...,z_k\rangle
\subset \mathcal{O}_{\C^n, \phi (a)} $,
\end{center}
and
\begin{center}
$z_1^N,...,z_k^N\in \langle f_1\circ\phi^{-1}, ...,
f_m\circ \phi^{-1}\rangle\subset \mathcal{O}_{\C^n, \phi(a)}. $
\end{center}
 Then, there exists an open set $V\subset U$ such that, in $\phi (V)$,
\begin{center}
	$(|f_1\circ\phi^{-1}|^2+...+|f_m\circ\phi^{-1}|^2)^M\leq C_1(|z_1|^2+...+|z_k|^2),$
\end{center}
and
\begin{center}
	$(|z_1|^2+...+|z_k|^2)^N\leq C_2(|f_1\circ\phi^{-1}|^2+...|f_m\circ\phi^{-1}|^2)$, 
\end{center}
where $C_1, C_2>0$.

Hence, there exist $C_3, C_4>0$ and an open neighbourhood $W\subset V$ of $a$ such that
\begin{center}
	$-C_3(-\log(|z_1|^2+...+|z_k|^2))^{\alpha}\leq v\circ\phi^{-1}\leq
	 -C_4(-\log(|z_1|^2+...+|z_k|^2))^{\alpha},$
\end{center}
in $\phi (W)$.

It follows from \cite{Blo06} that if $v_1\in\mathcal{D}, v_2\in PSH$ and $v_1\leq v_2$
then $v_2\in \mathcal{D}$. By Corollary \ref{cor},
we conclude that $v\circ\phi^{-1}\in\mathcal{D} (\phi (W))$ iff $0<\alpha<\frac{k}{n}$.
Hence $v\in \mathcal{D} (W)$ iff $0<\alpha<\frac{k}{n}$.

Moreover, it follows from \cite{Blo06} that belonging in $\mathcal{D}$ is a local property.
Thus $v\in \mathcal {D}(\Omega)$ iff $0<\alpha<\frac{k}{n}$.
\end{proof}	

\end{document}